\documentclass[12pt]{amsproc}

\theoremstyle{definition}

\theoremstyle{remark}

\numberwithin{equation}{section}
\usepackage{graphicx}
\usepackage{latexsym}
\usepackage{amsmath}
\usepackage{amssymb}
\usepackage{amsfonts}
\usepackage{verbatim}
\usepackage{mathrsfs}
\usepackage[modulo]{lineno} 
\usepackage[latin1]{inputenc}
\usepackage{color}
\usepackage[colorlinks,citecolor=blue,urlcolor=blue]{hyperref}

\newtheorem{tm}{Theorem}[section]

\newtheorem{ap}{Assumption}[section]

\newtheorem{lm}{Lemma}[section]
\newtheorem{cor}{Corollary}[section]
\newtheorem{ex}{Example}[section]

\newcommand{\ee}{\mathbb E}
\newcommand{\pp}{\mathbb P}

\newcommand{\rr}{\mathbb R}

\newcommand{\BB}{\mathcal B}
\newcommand{\CC}{\mathcal C}
\newcommand{\LL}{\mathcal L}

\newcommand{\OOO}{\mathscr O}
\newcommand{\FFF}{\mathscr F}

\newcommand{\<}{\langle}
\renewcommand{\>}{\rangle}
\allowdisplaybreaks \allowdisplaybreaks[4]

\begin{document}


\title[Asymptotic Log-Harnack Inequality for Monotone SPDE]
{Asymptotic Log-Harnack Inequality for Monotone SPDE with Multiplicative Noise}

\author{Zhihui Liu}
\address{Department of Mathematics, 
The Hong Kong University of Science and Technology, Kowloon, Hong Kong}
\curraddr{}
\email{zhliu@ust.hk and liuzhihui@lsec.cc.ac.cn}



\date{\today}

\dedicatory{}

\begin{abstract}
We derive an asymptotic log-Harnack inequality for nonlinear monotone SPDE driven by possibly degenerate multiplicative noise.
Our main tool is the asymptotic coupling by the change of measure.
As an application, we show that, under certain monotone and coercive conditions on the coefficients, the corresponding Markov semigroup is asymptotically strong Feller, asymptotic irreducibility, and possesses a unique and thus ergodic invariant measure.
The results are applied to highly degenerate finite-dimensional or infinite-dimensional diffusion processes.
\end{abstract}

\maketitle



\section{Introduction}
\label{sec1}

The dimension-free Harnack-type inequality has been a very efficient tool to study diffusion semigroups in recent years.
The dimension-free power-Harnack inequality and log-Harnack inequality were introduced in \cite{Wan97(PTRF)} for elliptic diffusion semigroups on noncompact Riemannian manifolds and \cite{Wan10(JMPA)} for heat semigroups on manifolds with boundary. 
Both inequalities have been investigated extensively and applied to SDEs and SPDEs via coupling by the change of measures, see, e.g., \cite{ATW09(SPA), Liu09(JEE), MRW11(PA), Wan07(AOP), Wan11(AOP), Wan17(JFA), WZ14(SPA), WZ15(PA), Zha10(PA)}, the monograph \cite{Wan13}, and references therein.
In particular, these inequalities imply gradient estimates and thus the strong Feller property, irreducibility, and the uniqueness of invariant probability measures of the associated Markov semigroups.

When the stochastic system is degenerate so that these properties are unavailable, power- and log-Harnack inequalities could not hold. 
In this case, it is natural to investigate weaker versions of these properties by exploiting Harnack inequalities in the weak version. 
For instance, the strong Feller property is invalid for stochastic 2D Navier--Stokes equations driven by degenerate additive noises, whereas the weaker version, called the asymptotically strong Feller property, had been proved in \cite{HM06(ANN)} and \cite{Xu11(JEE)} by making use of asymptotic couplings and a modified log-Harnack inequality, respectively. 
This type of log-Harnack inequality in the weak version is of the form
\begin{align} \label{alh}
P_t \log f(x) & \le \log P_t f(y)+\Phi(x, y)
+ \Psi_t(x, y) \|\nabla \log f\|_\infty,  
\end{align}
for any $t>0$, $x,y$ in the underlying Hilbert space $H$, and $f \in \BB^+_b(H)$ with $\|\nabla \log f\|_\infty<\infty$, where $\Phi$ and $\Psi_t : H \times H \rightarrow (0, \infty)$ are measurable with $\Psi_t \downarrow 0$ as $t \uparrow \infty$.

Now it is called an asymptotic log-Harnack inequality in \cite{BWY19(SPA)}, where the authors proved that the Markov semigroup $P_t$ is asymptotic strong Feller, asymptotic irreducibility, and possesses at most one invariant probability measure, provided that \eqref{alh} holds.
They also gave some applications to stochastic systems with infinite memory driven by non-degenerate noises.
Recently, such inequality \eqref{alh} was also studied in
\cite{LLX19(IDA), HLL20(JEE), HLL20(SPL)} for stochastic 3D fractional Leray-$\alpha$ model, semilinear SPDEs with monotone coefficients, and stochastic 2D hydrodynamical systems, with degenerate noises, respectively.

To the best of our knowledge, there are few results concerning the Harnack inequality for nonlinear monotone SPDEs driven by possibly degenerate multiplicative noise in the literature.
We only aware of \cite{HLL20(SPL)} considered a semilinear monotone SPDE driven by degenerate multiplicative noise using the coupling method and the strong dissipativity of the unbounded linear operator.
This is the main motivation of the present study. 
We also note that Harnack inequalities for nonlinear monotone SPDEs, including stochastic $p$-Laplacian equation and stochastic generalized porous media equation, both driven by non-degenerate additive noise were obtained in \cite{Liu09(JEE), LW08(JMAA), Wan07(AOP)}.

We focus on the possibly fully nonlinear SPDE \eqref{eq-z} under monotone assumptions on the coefficients; see Section \ref{sec2} for more details.
Inspired by the coupling method developed
in \cite{BKS20(AAP), BWY19(SPA), Hai02(PTRF), HMS11(PTRF), KS18(PTRF), Oda08(PTRF)}, we prove the asymptotic log-Harnack inequality \eqref{alh} and applied it to derive asymptotic properties for the corresponding Markov semigroup.
Several finite-dimensional and infinite-dimensional SDEs driven by highly degenerate noise are given to illustrate our main results.
In this setup, the strong Feller property is generally invalid, so we are in a weak situation without standard Wang-type log-Harnack inequalities.

The rest of the paper is organized as follows. 
In Section \ref{sec2}, we give some preliminaries for the considered Eq. \eqref{eq-z}. 
We preform and prove the first main result, the desired asymptotic log-Harnack inequality \eqref{har} in Theorem \ref{tm-har}, for Eq. \ref{eq-z} with non-degenerate noise in Section \ref{sec3}.
The idea is then extended, in Section \ref{sec4}, to the highly non-degenerate noise case (see Eq. \eqref{eq-xy}), where we get another main result, the asymptotic log-Harnack inequality \eqref{har+} in Theorem \ref{tm-har+}.
Finally, in the last section, we give some concrete examples either of SODEs or SPDEs.

\section{Preliminaries}
\label{sec2}

Let $V \subset H=H^* \subset V^*$ be a Gelfand triple, i.e., $(H, (\cdot, \cdot), \|\cdot\|)$ is a separable Hilbert space identified with its dual space $H^*$ via the Riesz isomorphism, $V$ is a reflexive Banach space continuously and densely embedded into $H$, and $V^*$ is the dual space of $V$ with respect to $H$.
Denote by $\<\cdot, \cdot\>$ the dualization between $V$ and $V^*$, then it follows that $\<u, v\>=(u, v)$ for any $u \in V$ and $v \in H$. 

Denote by $\BB_b(H)$ the class of bounded measurable functions on $H$ and $\BB^+_b(H)$ the set of positive functions in $\BB_b(H)$.
For a function $f \in \BB_b(H)$, define
\begin{align*}
|\nabla f|(x)=\limsup_{y \rightarrow x} \frac{|f(z)-f(w)|}{\|z-w\|}, 
\quad x \in H.
\end{align*}
Denote by $\|\cdot\|_\infty$ the uniform norm:
$\|\nabla f\|_\infty=\sup_{x \in H} |\nabla f|(x)$ and define 
${\rm Lip}(H)=\{ f : H \rightarrow \rr, \ \|\nabla f\|_\infty<\infty\}$, the family of all Lipschitz functions on $H$.
Set ${\rm Lip}_b(H):={\rm Lip}(H) \cap \BB_b(H)$.

Let $U$ be another separable Hilbert space and 
$(\LL_2(U, H), \|\cdot\|_{\LL_2})$ be the space consisting of all Hilbert--Schmidt operators from $U$ to $H$.
Let $(W_t)_{t \ge 0}$ be a $U$-valued cylindrical Wiener process with respect to a complete filtered probability space $(\Omega, \FFF, (\FFF_t)_{t \ge 0},\pp)$, 
i.e., there exists an orthonormal basis $\{e_n\}_{n=1}^\infty$ of $U$ and a family of independent 1D Brownian motions $\{\beta_n\}_{n=1}^\infty $ such that  
\begin{align*}
W_t=\sum_{n=1}^\infty e_n\beta_n(t),\quad t \ge 0.
\end{align*}

Let us consider the stochastic equation
\begin{align}\label{eq-z}
{\rm d}Z_t=b(Z) {\rm d}t+\sigma (Z_t) {\rm d}W_t,
\quad Z_0=z \in H,
\end{align}
where $b: V \rightarrow V^*$, $\sigma: V \rightarrow \LL_2(U; H)$ are measurable and satisfy the following conditions.

\begin{ap} \label{ap} 
There exist constants $\alpha>1$, $\eta \in \rr$, and $C_i>0$ with $i=1,2,3,4$ such that for all $u, v, w \in V$,
\begin{align} 
& \rr \ni c \mapsto  \<b(u+c v), w\>
\quad \text{is continuous}, \label{ap-con}  \\
& 2 \<b(u)-b(v), u-v\>
+\|\sigma(u)-\sigma(v)\|^2_{\LL_2} \le \eta \|u-v\|^2, \label{ap-mon} \\
& 2 \<b(w), w\>
+\|\sigma(w)\|^2_{\LL_2} 
\le C_1+\eta \|w\|^2-C_2 \|w\|^\alpha, \label{ap-coe} \\
& \|b(w)\|_{V^*} \le C_3+C_4\|w\|^{\alpha-1}. \label{ap-gro}
\end{align}  
\end{ap}

To derive an asymptotic log-Harnack inequality for Eq. \eqref{eq-z}, we need the following standard non-degenerate condition.

\begin{ap}\label{ap-ell}
$\sigma: H \rightarrow \LL_2(U; H)$ is bounded 
and invertible with bounded right pseudo-inverse $\sigma^{-1}: H \rightarrow \LL(H; U)$, i.e., $\sigma(z)\sigma^{-1}(z)={\rm Id}_H$ (the identity operator on $H$) for all $z \in H$, with $\|\sigma^{-1}\|_\infty:=\sup_{z \in H} \|\sigma^{-1}(z)\|_{\LL(H; U)}<\infty$. 
\end{ap}

Under the above monotone and coercive conditions, one has the following known well-posedness result of Eq. \eqref{eq-z} and the Markov property of the solution, see, e.g., \cite[Theorems II.2.1, II.2.2]{KR79} and \cite[Theorem 4.2.4, Proposition 4.3.5]{LR15}.

\begin{lm} \label{lm-well} 
Let $T>0$ and Assumption \ref{ap} hold.
For any $\FFF_0$-measurable $z \in L^2(\Omega, H)$, Eq. \eqref{eq-z} with initial datum $Z_0=z$ exists a unique solution $\{Z_t:\ t\in [0,T]\}$ in $L^2(\Omega; \CC([0,T]; H)) \cap L^\alpha(\Omega \times (0, T); V)$ which  is a Markov process such that 
\begin{align}
(Z_t, v)=(z, v)+\int_0^t \<b(Z_r), v\> {\rm d}r+\int_0^t (v, \sigma (Z_r) {\rm d}W_r)
\end{align}
holds a.s. for all $v \in V$ and $t \in [0, T]$.
\end{lm} 

 Denote by $(P_t)_{t \ge 0}$ the corresponding Markov semigroup, i.e.,  
\begin{align} \label{pt}
P_t f(z)=\ee [ f(Z^z_t) ], \quad 
t \ge 0, \ z \in H, \ f \in \BB_b(H).
\end{align}

\section{Asymptotic Log-Harnack Inequality}
\label{sec3}

Our main aim in this section is to derive an asymptotic log-Harnack inequality under Assumptions \ref{ap} and \ref{ap-ell}, and then use it to derive several asymptotic properties for $P_t$ defined in \eqref{pt}.

\subsection{Asymptotic Log-Harnack Inequality}

We first give the construction of an asymptotic coupling by the change of measure.

Let $\lambda>\eta/2$ be a constant, where $\eta$ appears in Eq. \eqref{ap-mon} and Eq. \eqref{ap-coe}.
Consider
\begin{align} \label{eq-cou}  
d \bar Z_t &=(b(\bar Z_t)
+\lambda \sigma(\bar Z_t) \sigma^{-1}(Z_t) (Z_t-\bar Z_t) ) {\rm d}t
+\sigma(\bar Z_t) {\rm d}W_t,  
\end{align}
with initial datum $\bar Z_0=\bar z \in H$. 
Under Assumptions \ref{ap}-\ref{ap-ell}, it is not difficult to check that the additional drift term $\lambda \sigma(\bar Z_t) \sigma^{-1}(Z_t) (Z_t-\bar Z_t)$ satisfies the hemicontinuity, locally monotonicity, coercivity, and growth conditions in \cite{LR10(JFA)}, thus the asymptotic coupling between $Z$ and $\bar Z$ is well-defined.

Our first aim is to examine that $\bar Z(t)$ has the same transition semigroup $P_t$ under another probability measure.
To this end, we set  
\begin{align} \label{v} 
v_t:=\lambda \sigma^{-1}(Z_t) (Z_t-\bar Z_t), \quad 
\widetilde{W}(t):=W_t +\int_0^t v_r {\rm d}r,
\end{align}
and define
\begin{align} \label{R}
R(t): & =\exp\Big( -\int_0^t \<v_r, {\rm d}W_r\>_U
-\frac12 \int_0^t \|v_r\|^2_U {\rm d}r\Big), \quad t \ge 0.
\end{align} 

We first check that $R$ defined by \eqref{R} is a local uniformly integrable martingale such that the following estimate \eqref{est-R} holds.

\begin{lm} \label{lm-R}
Under Assumption \ref{ap}, there exists a constant $\gamma=2 \lambda-\eta$ such that for any $T>0$,
\begin{align} \label{est-R} 
\sup_{t \in [0, T]} \ee[R(t) \log R(t)]  
\le \frac{\lambda^2 \|\sigma^{-1} \|_\infty^2}{2 \gamma} \|z-\bar z\|^2.
\end{align} 
Consequently, there exists a unique probability measure $\mathbb Q$ on 
$(\Omega, \FFF_\infty)$ such that
\begin{align} \label{Q} 
\frac{{\rm d}\mathbb Q|\FFF_t}{{\rm d} \pp|\FFF_t}=R(t), \quad t \ge 0.
\end{align} 
Moreover, $(\widetilde W_t)_{t \ge 0}$ is a $U$-valued cylindrical Wiener process under $\mathbb Q$.
\end{lm}

\begin{proof}
Let $T>0$ be fixed.
For any $n$ with $\|z\|<n$, define the stopping time 
\begin{align*} 
\tau_n=\inf \{ t \ge 0: \|Z_t\| \ge n\}.
\end{align*}  
Due to the non-explosion of Eq. \eqref{eq-z} and Eq. \eqref{eq-cou}, it is clear that $\tau_n \uparrow \infty$ as $n \uparrow \infty$ and $(Z_t)_{t \in [0, T \wedge \tau_n]}$ and $(\bar Z_t)_{t \in [0, T \wedge \tau_n]}$ are both bounded.
It follows from Assumption \ref{ap-ell} that Novikov condition holds on $[0, T \wedge \tau_n]$, i.e., 
\begin{align*} 
\ee \exp \Big(\frac12 \int_0^{T \wedge \tau_n} \|v_t\|^2_U {\rm d}t \Big) 
<\infty.
\end{align*}
According to Girsanov theorem, $({\widetilde W}_t)_{t \in [0,T \wedge \tau_n]}$ is a $U$-valued cylindrical Wiener Process under the probability measure ${\mathbb Q}_{T,n} := R(T \wedge \tau_n) \pp$.

By the construction \eqref{v}, we can rewrite Eq. \eqref{eq-z} and Eq. \eqref{eq-cou} on $T \wedge \tau_n$ as  
\begin{align} \label{eq-cou-rew}  
\begin{split}
d Z_t &=(b(Z_t)-\lambda (Z_t-\bar Z_t) ) {\rm d}t
+\sigma(Z_t) {\rm d}{\widetilde W}_t,  \quad t \le T \wedge \tau_n, \\
d \bar Z_t &=b(\bar Z_t) {\rm d}t +\sigma(\bar Z_t) {\rm d}{\widetilde W}_t,  \quad t \le T \wedge \tau_n,
\end{split}
\end{align}
with initial values $Z_0=z$ and $\bar Z_0=\bar z$, respectively.
Applying It\^o formula on $[0, T \wedge \tau_n]$, under the probability $\mathbb Q_{T,n}$, we have
\begin{align*} 
& {\rm d}\|Z_t-\bar Z_t\|^2 
=2 \<Z_t-\bar Z_t, (\sigma(Z_t)-\sigma(\bar Z_t) ) {\rm d}{\widetilde W}_t\> \\
& +\int_0^t [\|\sigma(Z_t)-\sigma(\bar Z_t)\|^2_{HS} 
+2\<Z_t-\bar Z_t, b(Z_t)-b(\bar Z_t)
-\lambda (Z_t-\bar Z_t) \> {\rm d}t.
\end{align*}
Taking expectations $\ee_{\mathbb Q_{T,n}}$ on the above equation and using the fact that $({\widetilde W}_t)_{t \in [0,T \wedge \tau_n]}$ is a cylindrical Wiener Process under ${\mathbb Q}_{T,n}$, we obtain 
\begin{align*} 
\ee_{\mathbb Q_{T,n}} \|Z_t-\bar Z_t\|^2  
\le \|z-\bar z\|^2 -(2 \lambda-\eta) \int_0^t \ee_{\mathbb Q_{T,n}}  \|Z_r-\bar Z_r\|^2 {\rm d}r, 
\end{align*} 
where we have used \eqref{ap-mon}.
The Gr\"onwall inequality leads to 
\begin{align} 
\ee_{\mathbb Q_{T,n}} \|Z_t-\bar Z_t\|^2 
\le e^{-\gamma t} \|z-\bar z\|^2, \quad  0 \le t \le T \wedge \tau_n,
\end{align}
with $\gamma=2 \lambda-\eta$.
It follows from the above estimate, Assumption \ref{ap-ell}, and Fubini theorem that 
\begin{align} \label{est-Rn}
& \sup_{t \in [0, T], n \ge 0} \ee[R(t \wedge \tau_n) \log R(t \wedge \tau_n)]  \nonumber \\
& =\sup_{t \in [0, T], n \ge 0} \ee_{\mathbb Q_{T,n}} [\log R(t \wedge \tau_n)]  \nonumber \\
& =\frac12 \sup_{n \ge 0} \int_0^{T \wedge \tau_n} \ee_{\mathbb Q_{T,n}} \|v_r\|^2_U {\rm d}r 
\le \frac{\lambda^2 \|\sigma^{-1} \|_\infty^2}{2 \gamma} \|z-\bar z\|^2.
\end{align}  

Let $0 \le s<t \le T$.
Using the dominated convergence theorem and the martingale property of
$(R(t \wedge \tau_n))_{t \in [0, T]}$, we have
\begin{align*} 
\ee[R(t)|\FFF_s]
=\ee[\lim_{n \rightarrow \infty} R(t \wedge \tau_n)|\FFF_s]
=R(s).
\end{align*}  
This shows that $(R(t))_{t \in [0, T]}$ is a martingale and thus 
${\mathbb Q}_T(A)={\mathbb Q}_{T,n}(A)$ for all $A \in \FFF_{T \wedge \tau_n}$, where ${\mathbb Q}_T := R(T) \pp$.
By Girsanov theorem, for any $T > 0$, $(\widetilde W_t)_{t \in [0, T]}$ is a cylindrical Wiener process under the probability measure ${\mathbb Q}_T $.  
By Fatou lemma,
\begin{align*} 
& \liminf_{n \rightarrow \infty} \ee_{\mathbb Q_{T,n}} [\log R(t \wedge \tau_n)] \\
& =\frac12 \liminf_{n \rightarrow \infty} \ee_{\mathbb Q_T}\int_0^{t \wedge \tau_n}  \|v_r\|^2_U {\rm d}r 
\ge \frac12 \ee_{\mathbb Q_T} \int_0^t \|v_r\|^2_U {\rm d}r,
\end{align*}  
and thus we get
\begin{align*}
& \sup_{t \in [0, T]} \ee[R(t) \log R(t)] 
=\sup_{t \in [0, T]} \ee_{\mathbb Q_T} [\log R(t)] \\
& = \frac12 \ee_{\mathbb Q_T} \int_0^T \|v_r\|^2_U {\rm d}r 
\le \liminf_{n \rightarrow \infty} \ee_{\mathbb Q_{T,n}} [\log R(t \wedge \tau_n)].
\end{align*} 
Then \eqref{est-R} follows from the estimate \eqref{est-Rn}. 

Finally, by the martingale property of $R$, the family $({\mathbb Q}_T)_{T>0}$ is harmonic.
By Kolmogorov harmonic theorem, there exists a unique probability measure ${\mathbb Q}$ on
$(\Omega, \FFF_\infty)$ such that \eqref{Q} holds, and thus $(\widetilde W_t)_{t \ge 0}$ is a U-valued cylindrical Wiener process under $\mathbb Q$.  
\end{proof}

Next, we show that $\|Z^z_t-\bar Z^{\bar z}_t\|$ decays exponentially fast as $t \rightarrow \infty$ in the $L^2(\Omega, \mathbb Q; H)$-norm sense, where $Z_\cdot^z$ and $\bar Z^{\bar z}_\cdot$ denote the solutions of Eq. \eqref{eq-z} with $Z_0=z$ and Eq. \eqref{eq-cou} with $\bar Z=\bar z$, respectively.

\begin{cor} \label{cor-x-y}
Under Assumption \ref{ap}, 
\begin{align} \label{x-y} 
\ee_{\mathbb Q} \|Z^z_t-\bar Z^{\bar z}_t\|^2 
\le e^{-\gamma t} \|z-\bar z\|^2,
\quad t \ge 0.
\end{align} 
\end{cor}

\begin{proof}
Let $t \ge 0$ be fixed.
Similarly to the proof in Lemma \ref{lm-R}, 
applying It\^o formula under the probability measure $\mathbb Q$ and using the condition \ref{ap-mon}, we obtain 
\begin{align*} 
\ee_{\mathbb Q} \|Z^z_t-\bar Z^{\bar z}_t\|^2  
\le \|z-\bar z\|^2 - \gamma \int_0^t \ee_{\mathbb Q}  \|Z^z_r-\bar Z^{\bar z}_r\|^2 {\rm d}r, 
\end{align*} 
with $\gamma=2 \lambda-\eta$, from which we conclude \eqref{x-y} by Gr\"onwall lemma.
 \end{proof}

From Lemma \ref{lm-R}, it is clear that $R$ can be also rewriten as
\begin{align} \label{R+}
R(t)=\exp\Big( -\int_0^t \<v_r, {\rm d}\widetilde{W}(r)\>_U
+\frac12 \int_0^t \|v_r\|^2_U {\rm d}r\Big), \quad t \ge 0,
\end{align} 
and is a uniformly integrable martingale under the probability measure $\mathbb Q$. 
Moreover, by the well-posedness of the asymptotic coupling \eqref{eq-cou}, or equivalently, Eq. \eqref{eq-cou-rew}, $\bar Z(t)$ has the same transition semigroup $P_t$ under the probability measure $\mathbb Q$ defined by \eqref{Q}:
\begin{align} \label{pt-y}
P_t f(\bar z)=\ee_{\mathbb Q} [f(\bar Z_t^{\bar z})],\quad 
t \ge 0, \ \bar z \in H, \ f \in \BB_b(H).
\end{align} 

Now we can give a proof of the asymptotic log-Harnack inequality \eqref{har} for Eq. \eqref{eq-z}.
Using the results in \cite[Theorem 2.1]{BWY19(SPA)}, we have the following asymptotic properties implied by this inequality.

\begin{tm} \label{tm-har}
Let Assumptions \ref{ap} and \ref{ap-ell} hold.
For any $t \ge 0$, $z, \bar z \in H$, and $f \in \BB^+_b(H)$ with 
$\|\nabla \log f\|_\infty<\infty$,
\begin{align}\label{har}
P_t \log f(z) & \le \log P_t f(\bar z)
+\frac{\lambda^2 \|\sigma^{-1} \|_\infty^2}{2 \gamma} \|z-\bar z\|^2 \nonumber \\
& \quad + e^{-\frac{\gamma t}2}\|\nabla \log f\|_\infty \|z-\bar z\|.
\end{align}
Consequently, the following asymptotic properties hold.
\begin{enumerate}
\item (Gradient estimate)
For any $t>0$ and $f \in {\rm Lip}_b(H)$, there exist a constant $C=\gamma^{-\frac12}\lambda \|\sigma^{-1}\|_\infty$ such that 
\begin{align}\label{est-gra}
\|D P_t f\|^2 \le C \sqrt{P_t f^2-(P_t f)^2}
+ e^{-\frac{\gamma t}2}\|\nabla f\|_\infty.
\end{align}
Consequently, $P_t$ is asymptotically strong Feller.

\item (Asymptotic irreducibility) 
Let $z \in H$ and $B \subset H$ be a measurable set such that
$\delta(z, B):=\liminf_{t \rightarrow \infty} P_t(z, B)>0$.
Then $\liminf_{t \rightarrow \infty} P_t(z, B_\epsilon)>0$ for all $z \in H$ and $\epsilon>0$, 
where $B_\epsilon:=\{z \in H; \inf_{w \in B} \|z-w\|<\epsilon\}$.
Moreover, for any $\epsilon_0 \in (0, \delta(z, B))$, there exists a constant $t_0>0$ such that
$P_t(w, B_\epsilon)>0$ provided $t \ge t_0$ such that
$e^{-\gamma t/2} \|z-w\|< \epsilon \epsilon_0$. 

\item (Ergodicity)
If $\eta<0$, 
then there exists a unique and thus ergodic invariant probability measure $\mu$ for $P_t$.
\end{enumerate}
\end{tm}

\begin{proof}
It follows from \eqref{pt-y} and \eqref{Q} that 
\begin{align*}
& P_t \log f(y) 
=\ee_{\mathbb Q} [\log f(\bar Z^{\bar z}_t)] \\
& =\ee_{\mathbb Q} [\log f(Z^z_t)]
+\ee_{\mathbb Q} [\log f(\bar Z^{\bar z}_t)-\log f(Z^z_t)] \\ 
& = \ee  [R(t) f(Z^z_t)]
+\ee_{\mathbb Q} [\log f(\bar Z^{\bar z}_t)-\log f(Z^z_t)].
\end{align*}
Using the Young inequality in \cite[Lemma 2.4]{ATW09(SPA)}, we get 
\begin{align*}
P_t \log f(y) 
& \le \log P_t f(x)+\ee [R(t) \log R(t)]
+ \|\nabla \log f\|_\infty \ee_{\mathbb Q} 
\|\bar Z^{\bar z}_t-Z^z_t\|.
\end{align*}
Taking into account \eqref{est-R} and \eqref{x-y}, we obtain \eqref{har}.

The gradient estimate \eqref{est-gra} and asymptotic irreducibility follow  from \cite{BWY19(SPA)}, Theorem 2.1 (1) and (4), respectively.
The asymptotically strong Feller property of $P_t$ is a direct consequence of the gradient estimate \eqref{est-gra} and \cite[Proposition 3.12]{HM06(ANN)}.
Then $P_t$ possesses at most one invariant measure.
To show the ergodicity, it suffices to show the existence of an invariant measure, which is proved in \cite[Theorem 4.3.9]{LR15}, so we complete the proof. 
\end{proof}

\section{Applications to Degenerate Diffusions}
\label{sec4}

In this section, our main aim is to generate the idea and results in Section \ref{sec3} to degenerate diffusions.
Let $(H_i, (\cdot, \cdot)_i, \|\cdot\|_i)$, $i=1,2$, be two separable Hilbert spaces and there are two Gelfand triples $V_i \subset H_i=H_i^* \subset V_i^*$, $i=1,2$. 
Define $H:=H_1 \times H_2$ with inner product
\begin{align} \label{pro-H} 
\Big( \Big(\begin{array}{c} x \\ y \end{array} \Big),
\Big(\begin{array}{c} \bar x \\ \bar y \end{array} \Big) \Big) 
:=(x, y)+(\bar x, \bar y), 
\quad (x, y), (\bar x, \bar y) \in H_1 \times H_2,
\end{align} 
and norm 
\begin{align}  \label{norm-H}
\Big\|\left(\begin{array}{c} x \\ y \end{array}\right) \Big\|
:=\sqrt{\|x\|_1^2+\|y\|_2^2}, \quad (x, y) \in H_1 \times H_2.
\end{align}  
Then $(H, (\cdot, \cdot), \|\cdot\|)$ is a separable Hilbert space.

Our main concern in this section is to consider 
\begin{align}\label{eq-xy}
\begin{split}
dX_t &=b_1(X_t, Y_t) {\rm d}t+\sigma_1 {\rm d}W_t,  \quad X_0=x \in H_1, \\
dY_t &=b_2(X_t, Y_t) {\rm d}t+\sigma_2 (X_t, Y_t) {\rm d}W_t,  
\quad Y_0=y \in H_2,
\end{split}
\end{align}
where $b_i: V \rightarrow V^*_1$, , $i=1,2$, $\sigma_1 \in \LL_2(U; H_1)$, and $\sigma_2: V_2 \rightarrow \LL_2(U; H_2)$, are measurable maps, and $W$ is a $U$-valued cylindrical Wiener process.
Denote by $\|\cdot\|_{\LL_2^i}$ the Hlibert--Schimdt operator norms from $U$ to $H_i$, $i=1,2$, respectively.
 
Eq. \eqref{eq-xy} can be rewritten as Eq. \eqref{eq-z} with initial datum
$Z_0=z=(x, y) \in H_1 \times H_2$, where 
\begin{align} \label{b-s}
Z=\left(\begin{array}{c} X \\ Y \end{array}\right), \quad
b(Z)=\left(\begin{array}{c} b_1(Z) \\ b_2(Z) \end{array}\right), \quad 
\sigma(Z)=\left(\begin{array}{c} \sigma_1 \\ \sigma_2(Z) \end{array}\right).
\end{align} 
It follows from the definition \eqref{norm-H} of the norm in $H$  that 
\begin{align*} 
\|\sigma(u)\|^2_{\LL_2}
& =\sum_{n=1}^\infty \Big\|\Big(\begin{array}{c} \sigma_1 e_k \\ \sigma_2(u) e_k \end{array} \Big) \Big]^2 
=\sum_{n=1}^\infty \|\sigma_1 e_k \|_1^2+\|\sigma_2(u) e_k \|_2^2  \\
& =\|\sigma_1\|^2_{\LL_2^1}+\|\sigma_2(u)\|^2_{\LL_2^2}, \quad u \in H.
\end{align*}  
This shows that, if \eqref{ap-mon}-\eqref{ap-coe} in Assumption \ref{ap} hold with different $\sigma$ and different constants, then Eq. \eqref{eq-xy} with coefficients \eqref{b-s} also satisfies Assumption \ref{ap}. 
Then one use the well-posedness result, Lemma \ref{lm-well} of Eq. \eqref{eq-xy}.

\begin{ap}\label{ap+}
\eqref{ap-mon}-\eqref{ap-gro} hold with $\sigma$, $C_j$, replaced by $\sigma_2$ (the corresponding $\|\cdot\|_{\LL_2}$-norms of $\sigma$ are replaced by $\|\cdot\|_{\LL^2_2}$-norm of $\sigma_2$), and $C_j+\|\sigma_1\|^2_{\LL_2^1}$ for $j=1,3$, respectively.
\end{ap}

Similarly to Assumption \ref{ap-ell}, we just give the following analogous non-degenerate condition on $\sigma_2$. 
The first diffusion $\sigma$ on $H_1$ may taken to be extremely degenerate.
In the examples given in the last section which are frequently used, one can choose $\sigma_1=0$.
 
\begin{ap}\label{ap-ell+}
$\sigma_2: H \rightarrow \LL_2(U; H_2)$ is bounded and invertible with bounded right pseudo-inverse $\sigma_2^{-1}: H \rightarrow \LL(H; U)$ with  $\|\sigma_2^{-1}\|_\infty:=\sup_{z \in H} \|\sigma_2^{-1}(z)\|_{\LL(H; U)}<\infty$.
\end{ap}  

In finite-dimensional case, this model has been intensively investigated, see, e.g., \cite{MSH02(SPA), Zha10(SPA)} for results on well-posedness, derivative formulas, ergodicity, Harnack inequalities, hypercontractivity, and so forth. 
We also note that \cite{Wan17(JFA)} studied Eq. \eqref{eq-xy} with linear $b_1$, semilinear $b_2$, degenerate $\sigma_1=0$, and non-degenerate 
$\sigma_2$ independent of the system.
They applied the coupling method in the semigroup framework to get a power-Harnack inequality and show the hypercontractivity of the Markov semigroup.  

For Eq. \eqref{eq-xy} with coefficients \eqref{b-s} satisfying Assumptions \ref{ap+} and \ref{ap-ell+}, we consider the asymptotic coupling  
\begin{align}  \label{eq-cou-xy} 
\begin{split} 
d \bar X_t &=(b_1(\bar X_t, \bar Y_t)+\lambda (X_t-\bar X_t) ){\rm d}t+\sigma_1 {\rm d}W_t,   \\
d \bar Y_t &=(b_2(\bar X_t, \bar Y_t)
+\sigma_2(\bar X_t, \bar Y_t) \widehat{v}_t ) {\rm d}t
+\sigma_2(\bar X_t, \bar Y_t) {\rm d}W_t,  
\end{split}
\end{align}
with initial datum $(\bar X_0, \bar Y_0)=(\bar x, \bar y) \in H_1 \times H_2$, where  
\begin{align} \label{v+}
\widehat{v}_t:=\lambda \sigma_2^{-1}(X_t, Y_t) (Y_t-\bar Y_t).
\end{align}
Under Assumptions \ref{ap+} and \ref{ap-ell+}, it is not difficult to check that the additional drift term $\sigma_2(\bar X_t, \bar Y_t) \widehat{v}_t$ with $v_t$ given by \eqref{v+} satisfies the hemicontinuity, locally monotonicity and growth condition in \cite{LR10(JFA)}, thus the above asymptotic coupling \eqref{eq-cou-xy} is well-defined.

Similarly to the arguments in Section \ref{sec3}, we set  
\begin{align*}   
\widehat{W}_t:=W_t +\int_0^t\widehat{v}( r) {\rm d}r,
\end{align*}
and define
\begin{align} \label{R+}
\widehat{R}(t): & =\exp\Big( -\int_0^t \<\widehat{v}_r, {\rm d}W_r\>_U
-\frac12 \int_0^t \|\widehat{v}_r\|^2_U {\rm d}r\Big), \quad t \ge 0.
\end{align} 
By using the stopping time technique and the dominated convergence theorem as in Lemma \ref{lm-R}, it is not difficult to show that $\widehat{R}$ defined by \eqref{R+} is a uniformly integrable martingale and that there exists a unique probability measure 
$\widehat{\mathbb Q}$ on 
$(\Omega, \FFF_\infty)$ such that
\begin{align} \label{Q+} 
\frac{{\rm d}\widehat{\mathbb Q}|\FFF_t}{{\rm d} \pp|\FFF_t}=\widehat{R}(t), \quad t \ge 0,
\end{align} 
and $(\widehat{W}_t)_{t \ge 0}$ is a cylindrical Wiener process under $\widehat{\mathbb Q}$.

Rewrite Eq. \eqref{eq-xy} and Eq. \eqref{eq-cou-xy} as   
\begin{align} \label{eq-xy+}
\begin{split}
dX_t &=(b_1(X_t, Y_t)- \sigma_1 v_t) {\rm d}t+\sigma_1 {\rm d}\widehat{W}_t,   \\
dY_t &=(b_2(X_t, Y_t)-\lambda (Y_t-\bar Y_t) ) {\rm d}r
+\sigma_2(X_t, Y_t) {\rm d}\widehat{W}_t,  
\end{split}
\end{align}
and
\begin{align} \label{eq-cou-xy+}
\begin{split}
d \bar X_t &=(b_1(\bar X_t, \bar Y_t)+\lambda (X_t-\bar X_t)
- \sigma_1 v_t ){\rm d}t+\sigma_1 {\rm d}\widehat{W}_t,   \\
d \bar Y_t &=(b_2(\bar X_t, \bar Y_t)) {\rm d}t
+\sigma_2(\bar X_t, \bar Y_t) {\rm d}\widehat{W}_t,  
\end{split}
\end{align}
with initial datum $(X_0, Y_0)=(x, y), (\bar X_0, \bar Y_0)=(\bar x, \bar y) \in H_1 \times H_2$, respectively. 
Denote by $Z_\cdot^z=(X_\cdot^x, Y_\cdot^y)$ and $\bar Z_\cdot^{\bar z}=(\bar X_\cdot^{\bar x}, \bar Y_\cdot^{\bar y})$ denote the solutions of Eq. \eqref{eq-xy+} with $z=(x, y)$ and Eq. \eqref{eq-cou-xy+} with $\bar z=(\bar x, \bar y)$, respectively.
The well-posedness of Eq. \eqref{eq-cou-xy+} implies that   
\begin{align} \label{pt-y+}
P_t f(\bar x, \bar y)=\ee_{\widehat{\mathbb Q}} [f(\bar X_t^{\bar x}, \bar Y_t^{\bar y})],\quad t \ge 0, \ f \in \BB_b(H_1 \times H_2).
\end{align}

We have the following uniform estimate of $\ee[\widehat{R}(t) \log \widehat{R}(t)]$ on any finite interval $[0, T]$ and exponential decay of $\|(X_t^x, Y_t^y)-(\bar X_t^{\bar x}, \bar Y_t^{\bar y})\|$ in the $L^2(\Omega, \widehat{\mathbb Q}; H_1 \times H_2)$-norm sense.
The details could be make rigorous by using stopping time technique, dominated convergence theorem, and Fatou lemma as in Lemma \ref{lm-R}.

\begin{lm} \label{lm-R+}
Under Assumptions \ref{ap+} and \ref{ap-ell+}, we have
\begin{align*} 
&\ee_{\widehat{\mathbb Q}} 
\|(X_t^x, Y_t^y)-(\bar X_t^{\bar x}, \bar Y_t^{\bar y})\|^2 
\le e^{-\gamma t} (\|x-\bar x\|^2+\|y-\bar y\|^2),
\quad t \ge 0,  \\
&\sup_{t \in [0, T]} \ee[\widehat{R}(t) \log \widehat{R}(t)]  
\le \frac{\lambda^2 \|\sigma_2^{-1}\|^2}{2 \gamma} (\|x-\bar x\|^2+\|y-\bar y\|^2),
\quad T>0.  
\end{align*} 
\end{lm}

\begin{proof}
Applying It\^o formula to $Z_t-\bar Z_t:=(X_t-\bar X_t, Y_t-\bar Y_t)$, we have
\begin{align*} 
& {\rm d}\|Z_t-\bar Z_t\|^2
=2 \<Z_t-\bar Z_t, (\sigma_2(Z_t)-\sigma_2(\bar Z_t) ) {\rm d}\widehat{W}_t\> \\
& +\int_0^t [\|\sigma_2(Z_t)-\sigma_2(\bar Z_t)\|^2_{\LL_2^2} 
+2\<Z_t-\bar Z_t, b(Z_t)-b(\bar Z_t)
-\lambda (Z_t-\bar Z_t) \> {\rm d}t.
\end{align*}
Taking expectations $\ee_{\widehat{\mathbb Q}}$ on the above equation, using the fact that $({\widehat W}_t)_{t \ge 0}$ is a cylindrical Wiener Process under $\widehat{\mathbb Q}$, and taking into account Assumption \ref{ap-mon} with $\sigma$ and related $\|\cdot\|_{\LL_2}$-norm  replaced by $\sigma_2$ and $\|\cdot\|_{\LL^2_2}$-norm of $\sigma_2$, respectively, we obtain 
\begin{align*} 
\ee_{\widehat{\mathbb Q}} \|Z_t-\bar Z_t\|^2  
\le \|z-\bar z\|^2 - \gamma \int_0^t \ee_{\widehat{\mathbb Q}}  \|Z_r-\bar Z_r\|^2 {\rm d}r, 
\end{align*} 
from which we conclude the first inequality.

On the other hand, it follows from Assumption \ref{ap-ell+} that 
\begin{align*}
& \sup_{t \in [0, T]} \ee[R(t) \log R(t)] 
=\sup_{t \in [0, T]} \ee_{\widehat{\mathbb Q}} [\log R(t)] \\
& = \frac12 \int_0^t \ee_{\widehat{\mathbb Q}} \|v_r\|^2_U {\rm d}r 
\le \frac{\lambda^2 \|\sigma_2^{-1}\|^2}{2 \gamma} \|z-\bar z\|^2.
\end{align*} 
This shows the last inequality and we complete the proof.
\end{proof}

Finally, we derive the following asymptotic log-Harnack inequality and asymptotic properties for Eq. \eqref{eq-xy}.  
The proof is similarly to that of Theorem \ref{tm-har}, so we omit the details.

\begin{tm} \label{tm-har+}
Let Assumptions \ref{ap+} and \ref{ap-ell+} hold.
For any $t \ge 0$, $(x, y), (\bar x, \bar y) \in H_1 \times H_2$, and $f \in \BB^+_b(H_1 \times H_2)$ with 
$\|\nabla \log f\|_\infty<\infty$,
\begin{align}\label{har+}
P_t \log f(\bar x, \bar y) & \le \log P_t f(x, y)
+\frac{\lambda^2 \|\sigma^{-1} \|_\infty^2}{2 \gamma} (\|x-\bar x\|^2+\|y-\bar y\|^2) \nonumber \\
& \quad + e^{-\frac{\gamma t}2}\|\nabla \log f\|_\infty 
\sqrt{\|x-\bar x\|^2+\|y-\bar y\|^2}.
\end{align}
Consequently, similar gradient estimate, asymptotically strong Feller, and asymptotic irreducibility in Theorem \ref{tm-har} holds.
In particular, if $\eta<0$, there exists a unique and thus ergodic invariant probability measure $\mu$ for $P_t$.
\end{tm}

\section{Examples}

\subsection{Asymptotic Log-Harnack Inequality for Degenerate SODEs}

For SODE with non-degenerate multiplicative noise, Harnack inequality was shown in \cite{Wan11(AOP)}.
Here we mainly focus on degenerate case in the framework of Section \ref{sec4}.
In finite-dimensional case that $H_1=\rr^n$ and $H_2=\rr^m$  with $n, m \ge 1$, $H=\rr^n \times \rr^m$ and both $V$, $H^*$, and $V^*$ coincide with $H$.
Then $W$ is an $m$-D Brownian motion.

\begin{ex}
Consider Eq. \eqref{eq-xy} in the case $n=m=1$ and 
\begin{align} \label{ex-HM}
b(X,Y)=\left(\begin{array}{c} -X \\ Y-Y^3 \end{array}\right), \quad 
\sigma(X,Y)=\left(\begin{array}{c} 0 \\ 1 \end{array}\right),
\quad X, Y \in \rr^2.
\end{align} 
It was shown in \cite[Example 3.14]{HM06(ANN)} that the corresponding Markov semigroup is not strong Feller.
Therefore, the standard log-Harnack inequality could not be valid. 

For $w=(w_1, w_2) \in \rr^2$, direct calculations yield that 
\begin{align} \label{ex-HM+} 
2\<b(w), w\>+\|\sigma(w)\|^2_{\LL_2}
=-2\|w\|^2-2(w_2^2-1)^2+3
\le -2\|w\|^2+3,
\end{align}
which shows \eqref{ap-coe} with $\eta=-2<0$.
Similarly, one can check other conditions in Assumptions \ref{ap} and \ref{ap-ell}, respectively, Assumptions \ref{ap+} and \ref{ap-ell+}, hold.
Then by Theorem \ref{tm-har+}, the associated Markov semigroup associated with Eq. \eqref{ex-HM} satisfies the asymptotic log-Harnack inequality \eqref{har} and possesses a unique ergodic invariant measure which is asymptotically strong Feller and irreducibility.
\end{ex}

\begin{ex}
Consider Eq. \eqref{eq-xy} driven by additive noise, i.e., $\sigma_2$ does not depend on the system \eqref{eq-xy}, with drift function $b$ satisfying
\begin{align} \label{ex-MSH}
\<b(w), w\> \le \alpha-\beta \|w\|^2, \quad w\in \rr^{n+m},
\end{align} 
for some constants $\alpha, \beta>0$.
This includes Eq. \eqref{eq-xy} with $b_i$ and $\sigma_i$ given by \eqref{ex-HM}.
So the strong Feller property fails and power- or log-Harnack inequality could not hold.
We note that the author in \cite[Theorem 4.4]{MSH02(SPA)} proved that the corresponding Markov semigroup possesses a unique ergodic invariant measure, provided the above dissipative condition holds for $b \in \CC^\infty(\rr^{m+n})$ and centain Lyapunov structure holds for Eq. \eqref{eq-xy}.

Using Theorem \ref{tm-har+}, we have that the associated Markov semigroup satisfies the asymptotic log-Harnack inequality \eqref{har} and possesses a unique ergodic invariant measure which is asymptotically strong Feller and irreducibility. 
We do not possess differentiable regularity on $b$ besides the monotone, coercive, and growth conditions in Assumption \ref{ap}. 
\end{ex}

\subsection{Asymptotic Log-Harnack Inequality for SPDEs}

In this part, we give several examples of SPDEs such that the main result, Theorems \ref{tm-har} and \ref{tm-har+}, can be applied.

Let $q \ge 2$, $\OOO$ be a bounded open subset of $\rr^d$, and $L^q$, $W_0^{1,q}$, and $W_0^{-1,q^*}$ be the usual Lebesgue and Sobolev spaces on $\OOO$, where $q^*=q/(q-1)$.
Then we have two Gelfand triples 
$W_0^{1,q} \subset L^2 \subset W_0^{-1,q^*}$ and $L^q \subset W_0^{-1,2} \subset L^{q^*}$.
Let $U=L^2$ and $W$ be an $L^2$-valued cylindrical Wiener process.

\subsubsection{Nondegenerate case}

In this part, we impose Assumption \ref{ap-ell} and the following Lipschitz continuity and linear growth conditions on $\sigma$: there exists a positive constant $L_\sigma$ such that for all $u, v, w \in L^2$,
\begin{align} \label{lip}
\|\sigma(u)-\sigma(v)\|^2_{\LL_2} \le L^2_\sigma \|u-v\|^2, \quad 
\|\sigma(w)\|^2_{\LL_2} 
\le L^2_\sigma(1+ \|w\|^2).  
\end{align} 

\begin{ex} 
Take $H=L^2$ and $V=W_0^{1,q}$. 
Consider the stochastic generalized $p$-Laplacian equation
\begin{align}\label{p-lap}
{\rm d}Z_t={\rm div}(|\nabla Z_t|^{q-2} \nabla Z_t- c |Z_t|^{\tilde{q}-2} Z_t) {\rm d}t+\sigma (Z_t) {\rm d}W_t, 
\end{align}
with $Z_0=z \in L^2$ and homogeneous Dirichlet boundary condition,
where $c \ge 0$ and $\tilde{q} \in [1, q]$.
Then Eq. \eqref{p-lap} is equivalent to Eq. \eqref{eq-z} with $b$ given by 
$b(u)={\rm div}(|\nabla u|^{q-2} \nabla u)- c |u|^{\tilde{q}-2} u$,
$u \in W_0^{1,q}$.
In this case, one can check that \ref{ap-con}--\ref{ap-gro} hold with $\sigma=0$; see, e.g., \cite[Examples 4.1.5 and 4.1.9]{LR15}.
Under the condition \eqref{lip}, we obtain Assumption \ref{ap} with 
$\alpha=q$ and $\eta=L^2_\sigma$.

We note that, for Eq. \eqref{p-lap} driven by non-degenerate additive noise, \cite[Theorem 1.1 and Example 3.3]{Liu09(JEE)} got a power-Harnack inequality to the related Markov semigroup $P_t$.
Whether $P_t$ satisfies power- or log-Harnack inequality in the multiplicative noise case is unknown.
Under Assumption \ref{ap-ell}, we use Theorem \ref{tm-har} to derive the asymptotic log-Harnack \eqref{har} with certain constants for $P_t$.

\end{ex}

\begin{ex} 
Take $H=W_0^{-1,2}$ and $V=L^q$. 
Consider the stochastic generalized porous media equation
\begin{align}\label{p-med}
{\rm d}Z_t=( L \Psi(Z_t) +\Phi(Z_t)){\rm d}t+\sigma (Z_t) {\rm d}W_t, 
\end{align}
with $Z_0=z \in W_0^{-1,2}$ and homogeneous Dirichlet boundary condition. 
Here $L$ is a negative definite self-adjoint linear operator in $L^2$ such that its inverse is bounded in $L^q$ (including the Dirichlet Laplacian operator), 
$\Psi, \Phi$ are Nemytskii operators related to functions $\psi, \phi: \rr \rightarrow \rr$, respectively, such that the following monotonicity and growth conditions hold for some constants $C_+>0$, $C, \eta \in \rr$: 
\begin{align*}
&|\psi(t)|+|\phi(t)-C t| \le C_+(1+|t|^{q-1}),  \\
& -2\<\Psi(u)-\Psi(v), u-v\>+2 \<\Phi(u)-\Phi(v), (-L)^{-1}(u-v)\> \\
& \le -C_+ \|u-v\|^q_{L^q}+\eta \|u-v\|^2,
\quad t \in \rr, \ u, v \in L^q.
\end{align*}

A very simple example satisfying the above two inequalities is give by 
$\psi(t)=|t|^{q-2} t$ and $\phi(t)=\eta t$, $t \in \rr$.
Then one can check that $b$ defined by 
$b(u)=L \Psi(u) +\Phi(u)$, $u \in W_0^{-1,2}$,
satisfies Assumption \ref{ap} with $\sigma=0$; see, e.g., \cite[Theorem A.2]{Wan07(AOP)}.
According to \eqref{lip}, we obtain Assumption \ref{ap} with 
$\alpha=q$ and $\eta=L^2_\sigma$.

We note that in the non-degenerate additive noise case, \cite[Theorem 1.1]{Wan07(AOP)} got a power-Harnack inequality to the related Markov semigroup $P_t$.
Whether power- or log-Harnack inequality holds for $P_t$ in 
the multiplicative noise case remains open.
Using Theorem \ref{tm-har}, under the usual non-degenerate Assumption \ref{ap-ell}, we conclude that the asymptotic log-Harnack \eqref{har} holds with certain constants for $P_t$.
\end{ex}

\subsubsection{Degenerate case}

Next we give an SPDE with degenerate multiplicative noise.

\begin{ex}
Let $\OOO=(0,1)$, $U=H_1=H_2=L^2$, and $W$ be an $L^2$-valued cylindrical Wiener process. 
Consider Eq. \eqref{eq-xy} with 
\begin{align} \label{ex-Hai}
b(X,Y)=\left(\begin{array}{c} \Delta X+Y+2X-X^3 \\ 
\Delta Y+X+2Y-Y^3\end{array}\right),
\quad (X,Y) \in L^2 \times L^2.
\end{align} 
We note that, the author in \cite[Theorem 6.2]{Hai02(PTRF)} showed that Eq. \eqref{eq-xy} with drift \eqref{ex-Hai} and diffusion given by 
\begin{align*}
\sigma(X,Y)=\left(\begin{array}{c} 0 \\ 
{\rm Id}_{L^2} \end{array}\right),
\quad (X,Y) \in L^2 \times L^2,
\end{align*} 
where ${\rm Id}_{L^2}$ denotes the identity operator on $L^2$, possesses a unique ergodic invariant measure, by using asymptotic coupling method in combination with the Lyapunov structure of this system. 

One can check that Assumption \ref{ap+} hold on $b$ given in \eqref{ex-Hai} and $\sigma$ given in \eqref{b-s} such that $\sigma \in \LL_2(U; H_1)$ and $\sigma_2$ satisfies \eqref{lip} with $\sigma$ and related $\|\cdot\|_{\LL_2}$-norm replaced by $\sigma_2$ and related $\|\cdot\|_{\LL^2_2}$-norm, respectively.
Then using Theorem \ref{tm-har+}, we show that the associated Markov semigroup satisfies the asymptotic log-Harnack inequality \eqref{har} and possesses a unique ergodic invariant measure.
\end{ex}


\bibliographystyle{amsalpha}
\bibliography{bib}

\end{document}